\documentclass[11pt,a4paper]{amsart}
\usepackage[utf8]{inputenc}
\usepackage{amsmath}
\usepackage{amsfonts}
\usepackage{amssymb}
\usepackage{graphicx}
\usepackage{amsthm}
\usepackage{enumerate}
\usepackage{mathrsfs}
\usepackage{color}
\usepackage{hyperref}
\hypersetup{
	colorlinks=true, %set true if you want colored links
	linktoc=all,     %set to all if you want both sections and subsections linked
	linkcolor=blue,  %choose some color if you want links to stand out
}
\def\A{\mathcal{A}}
\def\B{\mathcal{B}}
\def\M{\mathcal {M}}
\def\N{\mathbb {N}}
\def\H{\mathcal {H}}

\def\D{\mathbb D}
\def\F{\mathscr{F}}

\newcommand{\T}{\mathbb{T}}

\theoremstyle{plain}
\newtheorem{theorem}{Theorem}[section]
\newtheorem*{theorem*}{Theorem}
\newtheorem{proposition}[theorem]{\bf Proposition}

\newtheorem{lemma}[theorem]{Lemma}
\theoremstyle{definition}

\theoremstyle{remark}
\newtheorem{remark}[theorem]{\bf Remark}
\newtheorem{example}[theorem]{\bf Example}

\begin{document}
		
		\title[Homomorphisms between polydisk algebras]{Homomorphisms between algebras of holomorphic functions on the infinite polydisk}

		\author[V. Dimant]{Ver\'onica Dimant}
		
		\author[J. Singer]{Joaqu\'{\i}n Singer}

\thanks{Partially supported by Conicet PIP 11220130100483  and ANPCyT PICT 2015-2299 }

\subjclass[2010]{46J15, 46E50,  32A38, 30H05}
\keywords{spectrum, algebras of holomorphic functions, homomorphisms of algebras}

\address{Departamento de Matem\'{a}tica y Ciencias, Universidad de San
	Andr\'{e}s, Vito Dumas 284, (B1644BID) Victoria, Buenos Aires,
	Argentina and CONICET} \email{vero@udesa.edu.ar}

\address{Departamento de Matem\'{a}tica, Facultad de Ciencias Exactas y Naturales, Universidad de Buenos Aires, (1428) Buenos Aires,
	Argentina and IMAS-CONICET} \email{jsinger@dm.uba.ar}
		
\begin{abstract}
We study the vector-valued spectrum $\mathcal{M}_\infty(B_{c_0},B_{c_0})$, that is, the set of non null algebra homomorphisms from $\H^\infty(B_{c_0})$ to $\H^\infty(B_{c_0})$ which is naturally projected onto the closed unit ball of $\H^\infty(B_{c_0}, \ell_\infty)$, likewise the scalar-valued spectrum $\M_\infty(B_{c_0})$ which is projected over $\overline{B}_{\ell_\infty}$. Our itinerary begins in the scalar-valued spectrum $\mathcal{M}_\infty(B_{c_0})$: by expanding a result by Cole, Gamelin and Johnson (1992) we prove that on each fiber  there are $2^c$ disjoint analytic Gleason isometric copies  of $B_{\ell_\infty}$. For the vector-valued case, building on the previous result we obtain $2^c$ disjoint analytic Gleason isometric copies  of $B_{\mathcal{H}^\infty(B_{c_0},\ell_\infty)}$ on each fiber. We also take a look at the relationship between fibers and  Gleason parts for both vector-valued spectra $\mathcal{M}_{u,\infty}(B_{c_0},B_{c_0})$ and $\mathcal{M}_\infty(B_{c_0},B_{c_0})$.
\end{abstract}

		\maketitle
		\bigskip

\section*{Introduction}

In 1909, David Hilbert wrote an article \cite{Hilbert}, which is now considered as one of the pioneering works of the theory of analytic functions in infinitely many variables. There he claimed his intention of transferring the main theorems about analytic functions of several variables to the infinite dimensional setting. Following his path, the open unit ball of $c_0$ appears as a natural domain for analytic functions in infinitely many variables since it is the usual infinite dimensional version of the  polydisc $\D^n$. Even if infinite dimensional holomorphy has been developed along other directions than Hilbert's initial approach, the open unit ball of $c_0$ has continued to be a favorite domain for bounded holomorphic functions.

Our interest here is describing the set of homomorphisms between algebras of bounded holomorphic functions on the open unit ball of $c_0$. While pursuing this objective, we accomplish a description of the fibers of the spectrum of the algebra of bounded holomorphic functions on the open unit ball of $c_0$ which improves and completes results of \cite{ColeGamelinJohnson, AronFalcoGarciaMaestre, AronDimantLassalleMaestre}.
In order to explain more precisely our goal  we need to recall some definitions and introduce some notation.

For a complex Banach space $X$, let $\H^\infty(B_X)$ be the space of bounded holomorphic functions on $B_X$ (the open unit ball of $X$). The (scalar-valued) spectrum of this uniform algebra is the set $\M_\infty(B_X)=\{\varphi: \H^\infty(B_X)\to\mathbb C$ nonzero algebra homomorphisms$\}$. Since $X^*$ is contained in $\H^\infty(B_X)$, there is a natural projection $\pi: \M_\infty(B_X)\to\overline B_{X^{**}}$ given by $\pi(\varphi)=\varphi|_{X^*}$. Through this projection, $\M_\infty(B_X)$ has a fibered structure: for each $z\in \overline B_{X^{**}}$ the set $\pi^{-1}(z)=\{\varphi\in \M_\infty(B_X):\ \pi(\varphi)=z\}$ is called the \textit{fiber} over $z$. For an infinite dimensional space $X$, the study of the fibers of its spectrum has been performed in several articles, as \cite{AronColeGamelin, ColeGamelinJohnson, Farmer, AronFalcoGarciaMaestre}. Frequently, a useful sub-algebra of $\H^\infty(B_X)$ is considered: $\A_u(B_X)$ is the space of holomorphic functions on $B_X$ which are uniformly continuous. The spectrum  $\M_u(B_X)=\{\varphi: \A_u(B_X)\to\mathbb C$ nonzero algebra homomorphisms$\}$ is similarly projected onto $\overline B_{X^{**}}$ and the fibering over this set is consequently defined. By \cite{AronBerner, DavieGamelin}, there is a canonical extension $[f\mapsto \widetilde f]$  from $\H^\infty(B_X)$ to $\H^\infty(B_{X^{**}})$ which is an isometric homomorphism of Banach algebras.  In this way, for the spectrum $\M_\infty(B_X)$, in the fiber   over each $z\in B_{X^{**}}$  there is a distinguished element $\delta_z$ given by $\delta_z(f)=\widetilde f(z)$. The extension $[f\mapsto\widetilde f]$ is also an isometric algebra homomorphism from $\A_u(B_X)$ to $\A_u(B_{X^{**}})$. Hence, for every $f\in\A_u(B_X)$ the function $\widetilde f$ is uniformly continuous in $B_{X^{**}}$ so it can be  extended to $S_{X^{**}}$ (the sphere of $X^{**}$). Thus, for the spectrum $\M_u(B_X)$, there is a distinguished element $\delta_z$ in the fiber over $z$ for every $z\in \overline B_{X^{**}}$.

For $\A=\H^\infty(B_X)$ or $\A_u(B_X)$, the spectrum $\M(\A)$ is partitioned into equivalence classes called \textit{Gleason parts}. Since  $\M(\A)$ is contained in the sphere of $\A^*$ it is clear that $\|\varphi-\psi\|\le 2$, for every $\varphi,\ \psi\in \M(\A)$. For $\varphi \in \mathcal M(\mathcal A)$, the {\em Gleason part} of $\varphi$  is the set
\[\mathcal {GP}(\varphi) = \{ \psi : \ \rho(\varphi,\psi) < 1\}=\{\psi : \ \|\varphi - \psi\| < 2\},\]
where $\rho(\varphi, \psi) =\sup\{|\varphi(f)|:  \  f \in \A, \|f\| \leq 1, \psi(f) = 0 \}$.
In $\M(\A)$, the usual distance between the elements $\|\varphi-\psi\|$ is referred as the \textit{Gleason metric} and $\rho(\varphi, \psi)$ is called the \textit{ pseudo-hyperbolic distance}. The study of Gleason parts for $\M_\infty(B_X)$ and $\M_u(B_X)$ for infinite dimensional $X$ (with a special focus in the case $X=c_0$) was initiated in \cite{AronDimantLassalleMaestre}.

In \cite{DimantSinger} we consider, for complex Banach spaces $X$ and $Y$, the set $\M_\infty(B_X,B_Y)$ that we called \textit{vector-valued spectrum},
 defined by
\[
\M_\infty(B_X,B_Y)=\{ \Phi: \H^\infty(B_X)\to \H^\infty(B_Y) \textrm{ nonzero algebra homomorphisms}\}.
\] 

This vector-valued spectrum is fibered over the closed unit ball of $\H^\infty(B_Y, X^{**})=\{f:B_Y\to X^{**}$ bounded holomorphic functions$\}$ through the projection (with image $\overline B_{\H^\infty(B_Y, X^{**})}$)
\begin{align*}
			\xi \colon \M_\infty(B_X,B_Y) &\to \mathcal{H}^\infty(B_Y,X^{**}), \\
			\Phi & \mapsto \left[ y \mapsto [x^* \mapsto \Phi(x^*)(y)] \right].
		\end{align*} 
As in the scalar-valued case, this projection gives rise to a fibered structure of the spectrum. For each $g\in\overline B_{\H^\infty(B_Y, X^{**})}$, the \textit{fiber} over $g$ is the set 
$$
\mathscr{F}(g)=\{\Phi\in \M_\infty(B_X,B_Y):\ \xi(\Phi)=g\}.
$$
We can associate to each function $g\in \mathcal{H}^{\infty}(B_Y,X^{**}) $ satisfying $g(B_Y)\subset B_{X^{**}}$  a composition homomorphism $ C_g\in \M_\infty(B_X,B_Y)$ given by
\[
 C_g(f)=\widetilde f\circ g,\quad\textrm{ for all }f\in\H^\infty(B_X),
\] (where $\widetilde f\in \H^\infty(B_{X^{**}})$ is the canonical extension  of $f$ referred above). It is easy to see that $C_g$ belongs to $\mathscr{F}(g)$ and so each fiber over a function $g\in \mathcal{H}^{\infty}(B_Y,X^{**}) $ with $g(B_Y)\subset B_{X^{**}}$ has a distinguished element.

In a previous article \cite{DimantSinger} we studied the fibers for this spectrum and exhibited conditions assuring the  containment of analytic copies of  balls into some particular fibers. Now, focusing on the case $X=Y=c_0$ we could prove that on \textit{every} fiber we can insert infinitely many disjoint analytic copies of the unit ball of $\H^\infty(B_{c_0}, \ell_\infty)$, isometrically for the Gleason metric. To achieve this result we need first to produce a similar statement for the scalar-valued spectrum: there are infinitely many disjoint analytic Gleason isometric copies of $B_{\ell_\infty}$ into \textit{any} fiber of $\M_\infty(B_{c_0})$. This is our main result of Section \ref{caso-escalar}. It is done by expanding a construction due to Cole, Gamelin and Johnson in \cite[Th. 6.7]{ColeGamelinJohnson} and both gives a complete answer to a question posed in \cite{AronFalcoGarciaMaestre} and an improvement of  \cite[Cor. 3.12]{AronDimantLassalleMaestre}.
In Section \ref{caso-vectorial} the vector-valued version of this result is presented.
 
The notion of Gleason part has its version for the vector-valued spectrum. Indeed, for $\Phi\in\M_\infty(B_X, B_Y)$ we can define
\[\mathcal {GP}(\Phi) = \{\Psi : \ \|\Phi - \Psi\| < 2\}=\{ \Psi : \ \sigma(\Phi, \Psi) =\sup_{y\in B_Y}\rho(\delta_y\circ\Phi,\delta_y\circ\Psi) < 1\}.\]
As in the scalar-valued case, this leads to a partition of $\M_\infty(B_X, B_Y)$ into equivalence classes. This concept, with the name of \textit{norm vicinity} or \textit{path component}, was previously studied in several articles (see, for instance, \cite{AronGalindoLindstrom, chu-hugli-mackey,GalindoGamelinLindstrom,GorkinMortiniSuarez, HozokawaIzuchiZheng,MacCluerOhnoZhao}). In \cite{DimantSinger} we gave some relationships between fibers and Gleason parts for the spectrum $\M_\infty(B_X, B_Y)$. Now we continue in this line devoting ourselves to the case $X=Y=c_0$. Inspired by what is done in \cite{AronDimantLassalleMaestre} for the scalar-valued spectrum, we begin by studing Gleason parts for the simpler spectrum $$\M_{u,\infty}(B_{c_0},B_{c_0})=\{\Phi:\A_u(B_{c_0})\to \H^\infty (B_{c_0})\textrm{ nonzero algebra homomorphisms}\}.$$ 
Then, we look into Gleason parts for the more complex spectrum $\M_\infty(B_{c_0},B_{c_0})$. These are the topics of Section \ref{Gleason-parts}.

 For general theory about infinite dimensional holomorphy we refer the reader to the classical books  \cite{Dineen, MujicaLibro, Chae} or to the  recently published beautiful monograph  \cite{DirSer}. Note that in this last book the space $\H^\infty(B_{c_0})$ has a leading role: it is introduced in (and studied from) Chapter 2 while the theory of holomorphic functions on arbitrary Banach spaces is put off until Chapter 15. The exact quotation from Hilbert's article that we have referred at the beginning of this Introduction can be seen either in  \cite[p. 10]{DirSer} or in \cite[p. 188]{pietsch2007history}.

	\section{Fibers for the scalar-valued spectrum $\mathcal{M}_\infty(B_{c_0})$} \label{caso-escalar}

	We begin by describing the fibers of  the scalar-valued spectrum $\mathcal{M}_\infty(B_{c_0})$ because, as in \cite{DimantSinger}, we can profit from this knowledge in order to work in the vector-valued framework.
	
	In \cite[Th. 6.7]{ColeGamelinJohnson}, Cole, Gamelin and Johnson defined an analytic injection
		\begin{align*}
			\Phi: B_{\ell_\infty} \times B_{\ell_\infty} \to \mathcal{M}_\infty(B_{c_0}),
		\end{align*}
	such that the image of $\Phi(z,-)$ is contained in the fiber $\pi^{-1}(z)$. As a result, for each $z \in B_{\ell_\infty}$, the fiber over $z$ contains an analytic copy of $B_{\ell_\infty}$. Since the spectrum $\mathcal{M}_\infty(B_{c_0})$ is projected over $\overline B_{\ell_\infty}$, the question about whether the same result holds for the fibers over those $z \in S_{\ell_\infty}$ naturally arises. Aron, Falcó, García and Maestre \cite[Th. 2.2]{AronFalcoGarciaMaestre}, through a different construction managed to produce an analytic copy of $B_{\ell_\infty}$  into the fiber over $z$,
	for points $z$ in the \emph{infinite torus} $\T^\infty$ (that is, those $z$ satisfying $|z_n| = 1$ for all $n \in \mathbb{N}$). They  could also extend  this construction to boundary points $z$ having infinitely many coordinates with modulus 1. However, the question for boundary points with no (or finitely many) modulus 1 coordinates has remained open. Specifically, we can read in \cite[Rmk. 2.10]{AronFalcoGarciaMaestre} the following: ``... we do not know if $B_{\ell_\infty}$ can be embedded in the fiber over $z$, for $z$ in the unit sphere of $\ell_\infty$ but $|z_n|<1$ for all $n$, as for example $(\frac{n-1}{n})$''.

Our main result in this section answers this question and completes the picture of the fibers. We obtain it by modifying the construction of Cole, Gamelin and Johnson to reach points $z$ with a subsequence of coordinates at a positive distance from 1 and then, by means of  Gleason isometries between fibers, we manage to get to all $z$'s. In this way we prove that there is an analytic injection from $B_{\ell_\infty}$ into the fiber over $z$, for \textit{any} $z \in \overline{B}_{\ell_\infty}$. Moreover, as Cole, Gamelin and Johnson had proved for the fiber over 0, we obtain that the injection is in fact an isometry for the Gleason metric. 

It is important to comment that the above  result has simultaneously  been proved  by Choi, Falc\'o, Garc\'{\i}a, Jung and Maestre in \cite{ChoiFalcoGarciaJungMaestre}. Anyway, their argument is not the same, specifically the construction of the injection is different. For us, in order to apply this procedure to the vector-valued spectrum (see Theorem \ref{bola-en-fibra-vectorial}) it is relevant not only the statement of the theorem but also the construction performed in the proof of the theorem. In fact, we make use of our construction to prove the vector-valued result.

In addition, we can go a step further to obtain not just \textit{one} analytic  copy of the ball $B_{\ell_\infty}$ on each fiber but $2^c$ disjoint copies. Indeed, we prove that for any $z\in \overline{B}_{\ell_\infty}$ and for each $\eta\in\beta(\N)\setminus \N$ there is an analytic Gleason isometry $\Phi_z^\eta$ from $B_{\ell_\infty}$ into the fiber over $z$ and that any two of these copies are in different Gleason parts. 

We want to point out that a former version of the next theorem is referred in \cite[Rmk. 3.5]{AronDimantLassalleMaestre}. Precisely, at the moment that article was written, the statement of our Theorem \ref{ExtCGJ} was that there is \textit{one} analytic Gleason isometric copy of the ball $B_{\ell_\infty}$ on each fiber. Then, we have obtained this step forward with infinitely many disjoint balls on each fiber. Note also that this current statement turns out to be an improvement of \cite[Cor. 3.12]{AronDimantLassalleMaestre} because there it was proved the containment of $2^c$ \textit{discs} (instead of \textit{balls}) lying on different Gleason parts on the fiber over each $z$ in the \textit{open} (instead of \textit{closed}) unit ball of $\ell_\infty$.

	\begin{theorem}
		\label{ExtCGJ}
		For every $z \in \overline{B}_{\ell_\infty}$, there is a map
		\begin{align*}
		    \Phi_z: \beta(\N)\setminus \N \times B_{\ell_\infty} \to \pi^{-1}(z) \subset \mathcal{M}_\infty(B_{c_0}).
		\end{align*}
		satisfying
		\begin{enumerate}
		    \item For all $\eta \in \beta(\N) \setminus \N$, the mapping $\Phi_z^\eta:B_{\ell_\infty}\to \pi^{-1}(z)$ given by $\Phi_z^\eta(w) = \Phi_z(\eta,w)$ is an analytic Gleason isometry.
		    \item If $\eta_1 \not = \eta_2$ are in $\beta(\N) \setminus \N$, then the images $\Phi_z^{\eta_1}(B_{\ell_\infty})$ and $\Phi_z^{\eta_2}(B_{\ell_\infty})$ lie in different Gleason parts.
		\end{enumerate}
		\end{theorem}

The proof of this theorem is performed in several steps. To prove that our analytic mappings are isometric (from the Gleason metric of $B_{\ell_\infty}$ to the Gleason metric of the spectrum $\mathcal{M}_\infty(B_{c_0})$) and to deduce that images associated to different elements in $\beta(\N)\setminus \N$ lie in different Gleason parts we  make use of the following technical lemma regarding infinite products.
		\begin{lemma}
		\label{lematecnico}
				Let $(\alpha_k)$ be an increasing sequence of positive real numbers converging to 1, such that $\sum_j (1 - \alpha_j) < \infty$. Consider a point $z = (z_j)\in \overline{B}_{\ell_\infty}$ for which there exists $\delta>0$ satisfying $|z_j - 1| > \delta$ for all $j$. Then, for each $N \in \mathbb{N}$ such that $1-\alpha_j  < \delta/4$ for all $j \geq N$, the following infinite product converges (to a nonzero number):
		\begin{align}
		\label{product}
			\prod_{j=N}^{\infty}\frac{\alpha_j - z_j}{1 - \alpha_jz_j}.
		\end{align}

		Additionally, if $(\ell(k))_k$ is an increasing sequence of positive integers converging to $\infty$, then
		\begin{align}
		\label{limit}
		\lim_{k\to \infty} \prod_{j=N}^{\ell(k)}\frac{\alpha_j - \alpha_k z_j}{1-\alpha_j\alpha_kz_j} = \prod_{j=N}^{\infty}\frac{\alpha_j - z_j}{1 - \alpha_jz_j}.
		\end{align}		
	\end{lemma}
	\begin{proof}
		We first show that the product in \eqref{product} is convergent. Indeed, this is derived from the  inequality:
		\[
		\left| 1 - \frac{\alpha_j - z_j}{1 - \alpha_j z_j} \right| = \left|\frac{1 + z_j}{1 - \alpha_jz_j}\right|(1-\alpha_j)
	\leq \frac{2}{\delta3/4}(1-\alpha_j),
		\]
		where the last term is summable by our hypothesis.
		
		Secondly we study, for each $k \in \mathbb{N}$, the rate of convergence for the infinite product $\prod_{j=N}^{\infty}\frac{\alpha_j - \alpha_kz_j}{1-\alpha_j\alpha_kz_j}$. We have the following bound independent of $k$:
		\[
		\left|1 - \frac{\alpha_j - \alpha_kz_j}{1-\alpha_j\alpha_kz_j}  \right| = \left|\frac{1 + \alpha_kz_j}{1 - \alpha_j\alpha_kz_j}\right|(1-\alpha_j)
		\leq \frac{2}{\delta/2}(1-\alpha_j).
		\]
		Then, for any given $\varepsilon > 0$, there is a number $N_1 \in \mathbb{N}$ such that for every $\widetilde{N} \geq N_1$,
		\begin{align*}
		\left| \prod_{j=N}^{\infty}\frac{\alpha_j - \alpha_kz_j}{1 - \alpha_j\alpha_kz_j} - \prod_{j=N}^{\widetilde{N}}\frac{\alpha_j - \alpha_kz_j}{1 - \alpha_j\alpha_kz_j} \right| < \varepsilon/4,
		\end{align*}
		for all $k \in \mathbb{N}$, and
		\begin{align*}
		\left| \prod_{j=N}^{\infty}\frac{\alpha_j - z_j}{1 - \alpha_jz_j} - \prod_{j=N}^{\widetilde{N}}\frac{\alpha_j - z_j}{1 - \alpha_jz_j} \right| < \varepsilon/4.
		\end{align*}
		Now, we can find $k_0 \in \mathbb{N}$ such that for every $k \geq k_0$ we get
		\begin{align*}
		\left|\prod_{j=N}^{N_1}\frac{\alpha_j - \alpha_kz_j}{1-\alpha_j\alpha_kz_j} - \prod_{j=N}^{N_1}\frac{\alpha_j - z_j}{1 - \alpha_j z_j}\right| < \varepsilon/4.
		\end{align*}
		Finally, take $k_0'=\min \{k: \ell(k)\ge N_1\}$ and $k_1=\max \{k_0, k_0'\}$. Hence, for $k\ge k_1$,  the expression
		\begin{align*}
		\left|\prod_{j=N}^{\ell(k)}\frac{\alpha_j - \alpha_k z_j}{1 - \alpha_j \alpha_k z_j} - \prod_{j=N}^{\infty}\frac{\alpha_j - z_j}{1 - \alpha_jz_j}\right|
		\end{align*}
		is bounded by
		\begin{align*}
		\left|\prod_{j=N}^{\ell(k)}\frac{\alpha_j - \alpha_k z_j}{1 - \alpha_j \alpha_k z_j} - \prod_{j=N}^{\infty}\frac{\alpha_j - \alpha_k z_j}{1 - \alpha_j \alpha_k z_j}\right| + \left|\prod_{j=N}^{\infty}\frac{\alpha_j - \alpha_k z_j}{1 - \alpha_j \alpha_k z_j} - \prod_{j=N}^{N_1}\frac{\alpha_j - \alpha_k z_j}{1 - \alpha_j \alpha_k z_j}\right|\\
		+  \left|\prod_{j=N}^{N_1}\frac{\alpha_j - \alpha_k z_j}{1 - \alpha_j \alpha_k z_j} -  \prod_{j=N}^{N_1}\frac{\alpha_j - z_j}{1 - \alpha_j z_j}\right| + \left|\prod_{j=N}^{N_1}\frac{\alpha_j - z_j}{1 - \alpha_j z_j} -  \prod_{j=N}^{\infty}\frac{\alpha_j - z_j}{1 - \alpha_jz_j}\right| < \varepsilon.
		\end{align*}
		Thus,  the desired limit is proved.
	\end{proof}
	As we have already mentioned,  Lemma \ref{lematecnico} will be used to prove several Gleason isometries inside the fiber over $z$. Even if this lemma only works for points $z$ with coordinates at a positive distance from 1, we can overcome the restriction through the following remark.
	\begin{remark}
	\label{equalmodulus}		
		Let $z,w \in \overline{B}_{\ell_\infty}$ such that $|z_n|= |w_n|, \ \forall n \in \mathbb{N}$. Then,  the fibers  $\pi^{-1}(z)$ and $\pi^{-1}(w)$ are Gleason isometric.
	\end{remark}

Indeed, \cite[Prop. 1.6]{AronDimantLassalleMaestre} gives sufficient conditions for an automorphism $\Phi:B_X \to B_X$ to induce a Gleason isometry in the spectrum. In our particular setting, we can write $w_n = \lambda_n z_n$ for $\lambda_n \in \mathbb{T}$ and define the mapping
	
%	Indeed,  we can write $w_n = \lambda_n z_n$, for certain numbers $\lambda_n\in\mathbb T$ and apply \cite[Prop. 1.6]{AronDimantLassalleMaestre} to the following isomorphism:
		\begin{align*}
		\theta: \ell_\infty &\to \ell_\infty \\
		\theta(z) &= (\lambda_n z_n)_n.
		\end{align*}
Note that $\theta$ is an isometric linear mapping satisfying $\theta(B_{c_0}) = B_{c_0}$. Additionally, $\theta$ coincides with the bitraspose of its restriction to $c_0$. It follows that the induced mapping $\Lambda_\theta: \mathcal{M}_\infty(B_{c_0}) \to \mathcal{M}_\infty(B_{c_0})$ is an onto Gleason isometry taking the fiber over $z$ onto the fiber over $w$.

 	\begin{proof}[Proof of Theorem \ref{ExtCGJ}.]
 		Fix $(\alpha_k)$ an increasing sequence of positive real numbers converging to 1 such that  $\sum_k (1 - \alpha_k) < \infty$ and take  a sequence of non-negative integers $(m_k)_k$ satisfying $m_k + k < m_{k+1}$. We first prove the result under additional hypotheses over $z$. Then, we extend it to any $z \in \overline{B}_{\ell_\infty}$.
 		
 		Step 1:
 		Let  $z \in \overline{B}_{\ell_\infty}$ for which there exists $\delta > 0$ satisfying $|z_k - 1| > \delta$ for all $k \in \mathbb{N}$. Define, for each $k \in \mathbb{N}$, the mapping $\Phi_z^k: B_{\ell_\infty} \to B_{\ell_\infty}$ as follows
 		\[ 	(\Phi_z^k(w))_j = \left\{ \begin{array}{cc}
 			\alpha_k z_j, & \text{ for } 1 \leq j \leq k\\
            & \\
 			0,& \text{ for } k+1 \leq j \leq m_k \\
            & \\
 			\dfrac{\alpha_k - w_i}{1 - \alpha_kw_i},& \text{ for }  j = m_k + i, 1 \leq i \leq k\\
            & \\
 			0,& \text{ for } j > m_k + k. \end{array}\right.
 		\]
 By regarding each $\Phi_z^k(w)$ as an element in $\M_\infty(B_{c_0})$, we can now define
 \begin{align*}
     \Phi_z:\N \times B_{\ell_\infty} &\to \M_\infty(B_{c_0})\\
     \Phi_z(k,w) &= \Phi_z^k(w).
 \end{align*}
 Since $\M_\infty(B_{c_0})$ is a weak-star compact subset of $(\H^\infty(B_{c_0}))^*$, for each $w \in B_{\ell_\infty}$ there is a unique continuous extension
 \begin{align*}
    \Phi_z(-,w):\beta(\N) \to \M_\infty(B_{c_0}).
 \end{align*}
 We now write $\Phi_z^\eta(w) = \Phi_z(\eta,w).$ Note that, if $n \leq k$,
 \[\langle \Phi_z^k(w), e_n \rangle = \alpha_kz_n\underset{k\to\infty}{\longrightarrow} z_n.\]
 So $\pi(\Phi_z^\eta(w)) = z$ for every $w\in B_{\ell_\infty}$ and $\eta \in \beta(\N) \setminus \N$, meaning that the image of $\Phi_z^\eta$ is contained in $\pi^{-1}(z)$.
 To see that $\Phi_z^\eta$ is  an analytic mapping, note that, for a given $f \in \H^\infty(B_{c_0})$, the sequence $(f \circ \Phi_z^k)_k$ is contained in the weak-star compact set $\|f\|\overline B_{\H^\infty(B_{\ell_\infty})}$. A standard argument then yields that $f \circ \Phi_z^\eta$ is indeed in $\|f\|\overline B_{\H^\infty(B_{\ell_\infty})}$ and, in particular, holomorphic.
 
 Now we need to check that $\Phi_z^\eta$ is isometric. By \cite[Th. 2.4]{AronDimantLassalleMaestre} (see also \cite[(6.1)]{ColeGamelinJohnson}), we know that for all $w,v\in B_{\ell_\infty}$ and $ k\in\mathbb N$,
 \[ \label{rhodeltaphi}
 \rho (\delta_{\Phi_z^k(w)},\delta_{\Phi_z^k(v)})=\sup_{j\in \mathbb N} \left|\frac{\Phi_z^k(w)_j-\Phi_z^k(v)_j}{1-\overline{\Phi_z^k(w)}_j\Phi_z^k(v)_j}\right| = \sup_{1\le i\le k} \left|\frac{w_i-v_i}{1-\overline{w}_i v_i}\right|\le \rho(\delta_w, \delta_v).
 \]
 From this it is easily derived that
 \[
 \|\Phi_z^\eta(w)-\Phi_z^\eta(v)\|\le \| \delta_w-\delta_v\|\quad\textrm{for all }w,v\in B_{\ell_\infty},\  \eta\in\beta(\N)\setminus\N.
 \]
 To prove the reverse inequality we consider, for each $i\le N$, the mapping
 	\begin{align*}
 		G_{i,N}(\omega) = \prod_{j=N}^{\infty} \frac{\alpha_j - \omega_{i+m_j}}{1 - \alpha_j \omega_{i+m_j}}.
 	\end{align*}
 	A straightforward computation shows that the modulus of each factor is bounded by 1. Additionally, for every $ 0 < r < 1$ and $\omega \in rB_{c_0}$   we have that
 	\[ \left|1- \frac{\alpha_j - \omega_{i+m_j}}{1 - \alpha_j \omega_{i+m_j}} \right| = \left|\frac{1 + \omega_{i+m_j}}{1 - \alpha_j \omega_{i+m_j}}\right|(1 - \alpha_j)
 	\leq \frac{1 + r}{1 - r}(1 - \alpha_j).
 	\]
 	It follows that the partial products involved in $G_{i,N}$ converge uniformly on $rB_{c_0}$ for each $ 0 < r < 1$. Hence, applying a Weierstrass type theorem \cite[Th. 2.13]{DirSer} to the partial products, we get that $G_{i,N} \in \mathcal{H}^\infty(B_{c_0})$ with $\|G_{i,N}\|\le 1$, for all $i\le N $. We now consider for $k>N $ the composition $G_{i,N} (\Phi_z^k(w))$. Observe that whenever  $\lim_k G_{i,N} (\Phi_z^k(w))$ exists, it must coincide with $\Phi_z^\eta(w) (G_{i,N})$ for $\eta\in \beta(\N) \setminus \N$.
	If we denote with $\ell(k)$ the maximum $j \in \mathbb{N}$ such that $i + m_j \leq k$,  we can write
	\begin{align*}
		G_{i,N}(\Phi_z^k(w)) = \left(\prod_{j = N}^{\ell(k)} \frac{\alpha_j - \alpha_kz_{i+m_j}}{1 - \alpha_j\alpha_kz_{i+m_j}}\right) \cdot \left(\prod_{j = \ell(k) + 1}^{k-1}\alpha_j\right) \cdot w_i \cdot \left(\prod_{j= k+1}^{\infty}\alpha_j\right).
	\end{align*}
Since $\alpha_k\nearrow 1$ there exists $k_0\in\mathbb N$ such that $1-\alpha_k<\delta/4$ for every $k\ge k_0$. Then, if $N\ge k_0$, by Lemma \ref{lematecnico}, the first factor converges to $C_ {i,N}=\prod_{j=N}^{\infty}\frac{\alpha_j - z_{i+m_j}}{1 - \alpha_jz_{i+m_j}}$ as $k$ goes to infinity. Since both the second and the last terms converge to 1, we get for $\eta \in \beta(\N) \setminus \N$ that
		\[ \Phi_z^\eta(G_{i,N}) = C_{i,N}\cdot w_i. \]
Note that $|C_{i,N}|\le 1$ for all $N$ and $C_{i,N}\to 1$ when $N\to\infty$.
		 Thus, for each $h\in\mathcal H^\infty (\D)$ with $\|h\|\le 1$ we have that $h\circ G_{i,N}\in\mathcal H^\infty(B_{c_0})$ with $\|h\circ G_{i,N}\|\le 1$ and
\[
|\Phi_z^\eta(w)(h\circ G_{i,N})-\Phi_z^\eta(v)(h\circ G_{i,N})|=|h(C_{i,N}w_i)-h(C_{i,N}v_i)|\underset{N\to \infty}{\longrightarrow} |h(w_i)-h(v_i)|.
\]
This implies $\|\Phi_z^\eta(w)-\Phi_z^\eta(v)\|\ge \| \delta_{w_i}-\delta_{v_i}\|$ for all $i$, where the last norm is taken in $\M_\infty(\mathbb{D})$. Appealing once more to \cite[Th. 2.4]{AronDimantLassalleMaestre} or \cite[(6.1)]{ColeGamelinJohnson} we derive
\[
\|\Phi_z^\eta(w)-\Phi_z^\eta(v)\|\ge \| \delta_w-\delta_v\|\quad\textrm{for all }w,v\in B_{\ell_\infty},
\]	and hence the isometry is proved.

Finally, if $\eta_1, \eta_2$ are in $\beta(\N) \setminus \N$, with $\eta_1 \not = \eta_2$, then there exists an infinite set $A \subset \N$ such that $\N \setminus A$ is also infinite and $\eta_1 \in \overline A$, while $\eta_2 \in \overline{\N \setminus A}$. We now consider the function 
\begin{align*}
     G_{1,N}^A(\omega) = \prod_{\substack{j \geq N \\ j \in A}} \frac{\alpha_j - \omega_{1+m_j}}{1 - \alpha_j\omega_{1+m_j}}.
\end{align*}

Reasoning as above we can see for any $k \in A$, that $G_{1,N}^A(\Phi_z^k(0)) = 0$, so that $\Phi_z^{\eta_1}(0)(G_{1,N}^A) = 0.$ On the other hand, for $k'\geq N$ in $\N \setminus A$ we can compute 
\begin{align*}
    G_{1,N}^A(\Phi_z^{k'}(0)) = \prod_{\substack{N \leq j \leq \ell(k') \\ j \in A}} \frac{\alpha_j - \alpha_{k'}z_{1+m_j}}{1 - \alpha_j\alpha_{k'}z_{1+m_j}} \prod_{\substack{j \geq \ell(k')+1 \\ j \in A}} \alpha_j.
\end{align*}
And since the second factor converges to one, we obtain
\begin{align*}
    \Phi_z^{\eta_2}(0)(G_{1,N}^A) = \prod_{\substack{j \geq N \\ j \in A}} \frac{\alpha_j - z_{1+m_j}}{1 - \alpha_jz_{1+m_j}}.
\end{align*}
We then know by Lemma \ref{lematecnico} that $\Phi_z^{\eta_2}(0)(G_{1,N}^A) \to 1$ as $N \to \infty$, from which we conclude that $\Phi_z^{\eta_1}(0)$ and $\Phi_z^{\eta_2}(0)$ lie in different Gleason parts. The claim is finally derived from the fact that, by being a Gleason isometry, $\Phi_z^{\eta_i}(B_{\ell_\infty})$ is contained in a single Gleason part for $i = 1,2$.
\medskip
	
Step 2: Let us now consider $z \in \overline{B}_{\ell_\infty}$ for which there is $\delta > 0$ and a  subsequence $(z_{n_j})$ satisfying $|z_{n_j} - 1| > \delta$. We thus take $\mathbb{J} = \{n_j : j \in \mathbb{N} \}$ and perform the same construction as in the previous step restricting to those coordinates in $\mathbb{J}$. That is, for each $k \in \mathbb{N}$ we define  $\Phi_{z,\mathbb{J}}^k$ as
		\begin{align}
		\label{Phikajota}
		(\Phi_{z,\mathbb{J}}^k(w))_n = \left\{ \begin{array}{cc}
		\alpha_k z_{n}, & \text{ for } n \not= n_j\text{ for all } j,\ n \leq k, \\
               & \\
		\alpha_k z_{n_j}, & \text{ for } n = n_j, 1 \leq j \leq k,\\
               & \\
		\dfrac{\alpha_k - w_i}{1 - \alpha_kw_i},& \text{ for } n=n_j, \text{ with } j = m_k + i, 1 \leq i \leq k,\\
               & \\
		0,& \text{ otherwise. }
		\end{array}
		\right.
		\end{align}
	As before, we obtain  $\Phi_{z,\mathbb{J}}: \beta(\N) \times B_{\ell_\infty} \to \M_\infty(B_{c_0})$ by taking for each $w \in B_{\ell_\infty}$ the unique extension to $\beta(\N)$ of the mapping $\Phi_{z,\mathbb{J}}(-,w)$. The result is then achieved by considering the functions
		\[ G_{i,N}(\omega) = \prod_{j=N}^{\infty} \frac{\alpha_j - \omega_{n_{i+m_j}}}{1 - \alpha_j \omega_{n_{{i+m_j}}}}.\]
	Proceeding as in Step 1 we deduce, for each $\eta\in\beta(\N)\setminus \N$, that $\Phi_{z,\mathbb{J}}^\eta$ is an analytic Gleason isometry projecting over $z$. Furthermore, by considering the corresponding $G_{1,N}^A$ for this case it is readily seen that $\Phi_{z,\mathbb{J}}^{\eta_1}(B_{\ell_\infty})$ and $\Phi_{z,\mathbb{J}}^{\eta_2}(B_{\ell_\infty})$ lie in different Gleason parts whenever $\eta_1 \not = \eta_2$.

\medskip
	
Step 3: If $z \in \overline{B}_{\ell_\infty}$ does not satisfy the conditions of Step 2, then $z_k \to 1$ as $k \to \infty$. Fix $\lambda \in \mathbb{T}$, $\lambda \not = 1$. Then, we can apply the procedure of Step 1 to $\lambda z$ to get a mapping $\Phi_{\lambda z}:\beta(\N)\setminus \N\times B_{\ell_\infty}\to \pi^{-1}(\lambda z)$ with all the desired properties. Now, through Remark \ref{equalmodulus} the result follows.
 	\end{proof}	
 	
 \begin{remark}
 \label{gleasondeltaz}
 For $z \in B_{\ell_\infty}$, it is worth noting that the image  $\Phi_z^\eta(B_{\ell_\infty})$ is disjoint from $\mathcal{GP}(\delta_z)$. Indeed, we know from the proof of Theorem \ref{ExtCGJ} that $\Phi_z^\eta(0)(G_{1,N})= 0$, while $\delta_z (G_{i,N}) \to 1$ as $N \to \infty$, so that $\rho(\Phi_z^\eta(0),\delta_z) = 1$. The remark then yields from the fact that $\Psi_z^\eta$ is a Gleason isometry.
 \end{remark}

\section{Fibers for the vector-valued spectrum $\mathcal{M}_\infty(B_{c_0},B_{c_0})$} \label{caso-vectorial}

The vector-valued spectrum $\mathcal{M}_\infty(B_{c_0},B_{c_0})$ is projected onto $\overline B_{\H^\infty(B_{c_0}, \ell_\infty)}$ through the mapping $\xi$ given by $\xi(\Phi)(x)(x^*)=\Phi(x^*)(x)$ for all $x\in B_{c_0},\ x^*\in\ell_1$. Our aim is to prove that there are big analytic sets in the fibers given by this projection.
In \cite[Prop. 4.2, Th. 4.3,  Th. 4.6]{DimantSinger} we have shown different situations where $B_{\H^\infty(B_Y)}$ is analytically injected into the fiber $\F(g)\subset \M_\infty(B_X,B_Y)$ for \textit{constant} functions $g$. Also, in \cite[Th. 4.5]{DimantSinger} we have seen that, if there exists a polynomial on $X$ not weakly continuous on bounded sets, then the complex disk $\D$ is analytically inserted in $\F(g)\subset \M_\infty(B_X,B_Y)$ for every $g\in B_{\H^\infty(B_Y, X^{**})}$. But $c_0$ is the typical example of space where all the polynomials are weakly continuous on bounded sets. Nevertheless, we will see in Theorem \ref{bola-en-fibra-vectorial} that the ball $B_{\H^\infty(B_{c_0},\ell_\infty)}$ can be analytically injected in the fiber $\F(g)\subset \M_\infty(B_{c_0},B_{c_0})$ for \textit{every} $g\in \overline B_{\H^\infty(B_{c_0}, \ell_\infty)}$. Moreover, as in the scalar-valued case, we produce $2^c$ analytic Gleason isometric copies of $B_{\H^\infty(B_{c_0},\ell_\infty)}$  in the fiber $\F(g)$ with each of these copies lying in different Gleason parts. 
	
As we previously did in \cite{DimantSinger}, the vector-valued result will be performed by building on the mappings obtained in the scalar-valued case.  Note that  $\Phi\in\F(g)$ if and only if $\delta_x\circ\Phi\in\pi^{-1}(g(x))$ for all $x\in B_{c_0}$. Hence in order to have, for each $\eta\in\beta(\N)\setminus \N$, an analytic mapping $\Psi^\eta$ from $B_{\H^\infty(B_{c_0},\ell_\infty)}$  into the fiber $\F(g)$ it seems natural to propose
\begin{equation} \label{Psi-propuesta}
    \Psi^\eta(h)(f)(x)=\Phi_{g(x)}^\eta(h(x))(f).
\end{equation}
The problem is that the construction of $\Phi^\eta_z$ in Theorem \ref{ExtCGJ} is dependent on $z$ (specifically, whether $z$ has a subsequence whose coordinates are far away from 1 and which of those coordinates are used). Thus, to make formula \eqref{Psi-propuesta} work we need all $z=g(x)$ to be of the \textit{same} kind, independently of $x\in B_{c_0}$. In other words, we need (perhaps after a rotation) a subsequence $(g(x))_{n_k}$, independent of $x$ with all of its elements at a positive distance from 1. Before proving the existance of such subsequence, let us first recall some information about the functions $g$ in $\overline{B}_{\mathcal{H}^\infty(B_{c_0},\ell_\infty)}$.
	
\begin{remark}\label{Dos-opciones}
Let $g \in \overline{B}_{\mathcal{H}^\infty(B_{c_0},\ell_\infty)}$. If there exists $x_0\in B_{c_0}$ such that $\|g(x_0)\|=1$ then $\|g(x)\|=1$ for all $x\in B_{c_0}$. Hence,   there are two alternatives for  the range of $g$:
		\begin{itemize}
			\item $g(B_{c_0}) \subset B_{\ell_\infty}$.
			\item $g(B_{c_0}) \subset S_{\ell_\infty}$.
		\end{itemize}
	\end{remark}

 A function  $g\in\overline{B}_{\mathcal{H}^\infty(B_{c_0},\ell_\infty)}$ can be viewed as a sequence of functions $g= (g_n)_n$ with each $g_n\in\overline{B}_{\mathcal{H}^\infty(B_{c_0})}$. Let us see now that the expected behaviour of $g$  actually holds.

\begin{lemma}	\label{Bingo}
Let $g \in \overline{B}_{\mathcal{H}^\infty(B_{c_0},\ell_\infty)}$, $g= (g_n)_n$. Then at least one of the following occur:
	\begin{enumerate}
		\item[(i)] There exists a subsequence $(g_{n_k})$ such that $g' = (g_{n_k})_k$ satisfies $g'(B_{c_0}) \subset B_{\ell_\infty}$.
			\item[(ii)] There exist $\lambda \in \mathbb{T}$ and a subsequence $(g_{n_k})$ such that   $g_{n_k}(x) \to \lambda$ for all $x \in B_{c_0}$.
		\end{enumerate}
	\end{lemma}

	\begin{proof} Applying \cite[Th. 2.17]{DirSer} for the sequence $(g_n)$ we know that there  exist a subsequence $(g_{n_k})_k$ and a function $h\in \overline{B}_{\mathcal{H}^\infty(B_{c_0})}$  such that $g_{n_k}(x)\to h(x)$, for all $x\in B_{c_0}$. If $g$ does not satisfy $(i)$, taking into account what we comment in Remark \ref{Dos-opciones}, we realize that
		$$
		\sup_{k\ge k_0} |g_{n_k}(x)|=1\ \textrm{ for all } x\in B_{c_0},\textrm{ and } k_0\in\mathbb{N}.
		$$
		This implies that there is $ \lambda \in \mathbb{T}$ such that $h(x)=\lambda$, for all $x$ and thus $ (g_{n_k})_k$ satisfies $(ii)$.
	\end{proof}

We also need to have an isometry between certain fibers of the vector-valued spectrum, in the spirit of what we stated in Remark \ref{equalmodulus} for the scalar-valued case.

	\begin{lemma} \label{MismoModulo}
		Let $g,h \in \overline B_{\mathcal{H}^\infty(B_{c_0}, \ell_\infty)}$ such that $|g_n(x)| = |h_n(x)|$, for all $  n \in \mathbb{N}$ and  $ x \in B_{c_0}$. Then:
\begin{enumerate}
			\item[(i)] There exists a sequence $(\lambda_n)$ in $\mathbb T$ such that $h_n(x) = \lambda_n g_n(x)$ for all $  n \in \mathbb{N}$ and $x \in B_{c_0}$.
			\item[(ii)]  The fibers $\mathscr{F}(g)$ and $\mathscr{F}(h)$ are Gleason isometric.
		\end{enumerate}
	\end{lemma}
	
	\begin{proof} $(i)$ For a given $n$, if $g_n \equiv 0$ the result is trivial. If not,  we can find an open subset $V_n$ of $B_{c_0}$ such that $g_n \not = 0$ on $V_n$. This implies that $\frac{h_n}{g_n}$ is a holomorphic function in $V_n$ having constant modulus $1$. This in turn implies that this function is constant and so $(i)$ holds.

$(ii)$ Given numbers $(\lambda_n)$ in $\mathbb T$, let $\theta: B_{c_0} \to B_{c_0}$ be the mapping defined by $\theta(x_n) = (\lambda_n x_n)$. A similar argument as the one used in Remark \ref{equalmodulus}, this time using  \cite[Prop. 5.3]{DimantSinger} instead of \cite[Prop. 1.6]{AronDimantLassalleMaestre} yields the desired result.
	\end{proof}
	
Now, we have all the necessary ingredients to proceed with the promised result.
	
\begin{theorem} \label{bola-en-fibra-vectorial}
	For every $g \in \overline{B}_{\mathcal{H}^\infty(B_{c_0},\ell_\infty)}$, there is a map 
	\begin{align*}\Psi_g: \beta(\N)\setminus \N \times B_{\mathcal{H}^\infty(B_{c_0}, \ell_\infty)} \to \mathscr{F}(g),
	\end{align*}
	satisfying 
	\begin{enumerate}
	    \item For all $\eta \in \beta(\N) \setminus \N$, the mapping $\Psi_g^\eta:B_{\mathcal{H}^\infty(B_{c_0}, \ell_\infty)} \to \mathscr{F}(g)$, given by $\Psi_g^\eta(h) = \Psi_g(\eta,h).$ is an analytic Gleason isometry.
        \item If $\eta_1 \not = \eta_2$ are in $\beta(\N) \setminus \N$, then the images $\Psi_g^{\eta_1}(B_{\mathcal{H}^\infty(B_{c_0}, \ell_\infty)})$ and $\Psi_g^{\eta_2}(B_{\mathcal{H}^\infty(B_{c_0}, \ell_\infty)})$ lie in different Gleason parts.
	\end{enumerate}
	\end{theorem}
	
	\begin{proof}
	Let $g \in \overline{B}_{\mathcal{H}^\infty(B_{c_0},\ell_\infty)}$, $g = (g_n)_n$. We first assume that $(g_n)$ has a subsequence $(g_{n_k})$ such that $g' = (g_{n_k})_k$ verifies $g'(B_{c_0}) \subset B_{\ell_\infty}$. We write $ \mathbb{J} = \{n_k\}_k$, take the functions $\Phi^k_{z,\mathbb{J}}(w)$ from \eqref{Phikajota} and define 
	    \begin{align*}
	        \Psi_g:\mathbb{N} \times B_{\mathcal{H}^\infty(B_{c_0}, \ell_\infty)} &\to \M_\infty(B_{c_0},B_{c_0})\\
	        \Psi_g(k,h)(f)(x) &= \Phi_{g(x),\mathbb{J}}^k(h(x))(f).
	    \end{align*}
	For each $k \in \N$, the mapping $[(x,x') \mapsto \Phi_{g(x),\mathbb{J}}^k(h(x'))]$ is separately holomorphic in $B_{c_0} \times B_{c_0}$, and, by Hartogs' Theorem, it is holomorphic in  $B_{c_0} \times B_{c_0}$. By restricting the previous mapping to the diagonal we obtain that $\Psi_g^k(h)(f)=\Psi_g(k,h)(f)$ is holomorphic. A straightforward computation shows that it is also bounded and that $\Psi_g^k$ is well-defined.
	By the weak-star compactness of $\M_\infty(B_{c_0},B_{c_0})$ (see for instance \cite[p. 3]{DimantSinger}), for each fixed $h \in B_{\mathcal{H}^\infty(B_{c_0}, \ell_\infty)}$, the mapping $\Psi_g(-,h):\N \to \M_\infty(B_{c_0},B_{c_0})$ has a unique extension to $\beta(\N)$; this induces
	    \begin{align*}
	        \Psi_g:\beta(\mathbb{N}) \times B_{\mathcal{H}^\infty(B_{c_0}, \ell_\infty)} &\to \M_\infty(B_{c_0},B_{c_0}).
	     \end{align*}
Note that for every $\eta \in \beta(\N) \setminus \N$, and $x \in B_{c_0}$ we have the following equality 
\begin{align}
\label{deltax}
    \delta_x \circ \Psi_g^\eta(h) = \Phi_{g(x),\mathbb{J}}^\eta(h(x)).
\end{align}
Since we know from Theoren \ref{ExtCGJ} that $\pi(\Phi_{g(x),\mathbb{J}}^\eta(h(x))) = g(x)$, the previous equality yields that $\xi(\Psi_g^\eta(h)) =g$ and thus the image of $\Psi_g^\eta$ is contained in $\mathscr{F}(g)$.
It is also readily seen that $\Psi_g^\eta$ is a Gleason isometry, as we have that
	\begin{align*}
		\rho(\delta_x \circ \Psi_g^\eta(h), \delta_x \circ \Psi_g^\eta (h')) = \rho(\Phi_{g(x),\mathbb{J}}^\eta(h(x)), \Phi_{g(x),\mathbb{J}}^\eta(h'(x))) = \rho (\delta_{h(x)}, \delta_{h'(x)}),
	\end{align*}
	where the last equality derives from the fact that, by Theorem \ref{ExtCGJ}, $\Phi_{g(x),\mathbb{J}}^\eta$ is a Gleason isometry. If we now take $\eta_1 \not = \eta_2 \in \beta(\N) \setminus \N$, it follows from \eqref{deltax} paired with Theorem \ref{ExtCGJ} that the images $\Phi_g^{\eta_1}(B_{\mathcal{H}^\infty(B_{c_0}, \ell_\infty)})$ and $\Phi_g^{\eta_2}(B_{\mathcal{H}^\infty(B_{c_0}, \ell_\infty)})$ lie in different Gleason parts.

	Additionally, for any fixed $f \in \H^\infty(B_{c_0})$, $x \in B_{c_0}$ and $\eta \in \beta (\N) \setminus \N$, the equality in \eqref{deltax} shows that the mapping $\delta_x \circ f \circ \Psi_g: B_{\mathcal{H}^\infty(B_{c_0}, \ell_\infty)} \to \mathbb{C}$ can be seen as the composition of the following analytic mappings
	\[\begin{array}{rclrcl}
    B_{\H^\infty(B_{c_0}, \ell_\infty)} &\to & B_{\ell_\infty} &\qquad\qquad B_{\ell_\infty} &\to &\mathbb C\\
    h &\mapsto & h(x) &\qquad\qquad z &\mapsto & \Phi_{g(x),\mathbb{J}}^\eta(z)(f),
    \end{array}
    \]
    showing that $\Psi_g^\eta$ is analytic.

	If instead there is no subsequence such that the corresponding $g'$ satisfies $g'(B_{c_0}) \subseteq B_{\ell_\infty}$, then by Lemma \ref{Bingo}, we can find a subsequence $(g_{n_k})$ and a number $\lambda \in \mathbb{T}$ such that $g_{n_k}(x) \to \lambda$ for every $x \in B_{c_0}$. If $\lambda\not= 1$, for each $x \in B_{c_0}$ there exist $\delta = \delta(x) > 0$ and $k_0=k_0(x)\in\mathbb N$ such that $ |g_{n_k}(x) - 1| > \delta$ for all $k\ge k_0$. Let $\mathbb J  = \{n_k : k \in \N \}$. Taking then
	\begin{align*}
		\Psi_g^k(h)(f)(x) = \Phi_{g(x),\mathbb{J}}^k(h(x))(f),
	\end{align*}
	and proceeding as before we obtain the desired result.
	Finally, the remaining case, that is, $\lambda = 1$, follows from the previous argument combined with Lemma \ref{MismoModulo}.
 	\end{proof}
	
	%\begin{lemma}[\txcr{Mover a partes de Gleason?}]
%		For any given $g \in \overline{B}_{\mathcal{H}^\infty(B_{c_0},\ell_\infty)}$, the image of the mapping
%		\begin{align*}
%		\Psi_g: B_{\mathcal{H}^\infty(B_{c_0},\ell_\infty)} &\to \mathscr{F}(g)\\
%		h &\mapsto \Psi(g,h).
%		\end{align*}
%		Is contained in a single Gleason part.
%	\end{lemma}
%%	
%	\begin{proof}
%		Note that for every $k \in \mathbb{N}$ we have that
%		\begin{align*}
%			\|\Phi_k(z,w) - \Phi(z,w')\| \leq \|\delta_{(z,w)} - \delta_{(z,w')} \|.
%		\end{align*}
%		It follows that $\| \Phi(z,w) - \Phi(z,w')\| \leq \|\delta_{(z,w)} - \delta_{(z,w')} \|.$
%	\end{proof}

	\begin{remark}
		\label{TaylorFalco}
In \cite[Lem. 2.9]{AronFalcoGarciaMaestre} it was shown that, for each $b\in\mathbb D$ and $w\in \overline B_{\ell_\infty}$, the fibers over $w$ and $(b,w)$ are homeomorphic. Then, it is observed in \cite{AronDimantLassalleMaestre} that this homeomorphism is in fact, a Gleason isometry.
		
    In order to present an extension of this result to the vector-valued spectrum, let us recall how this scalar-valued isometry is built. For that, given $b\in\mathbb D$, let us denote by $\Lambda_b:B_{c_0}\to B_{c_0}$ the mapping given by $\Lambda_b(x)=(b,x)$. Also, we denote by $S$ the left shift operator from $c_0$ onto $c_0$ (that is, $Sx=(x_2,x_3,\dots)$).

    The construction is the following: any $\varphi$ in the fiber over $w$ is mapped to $\psi$ in the fiber over $(b,w)$, where $\psi(f)=\varphi(f\circ \Lambda_b)$ for every $f\in\mathcal H^\infty(B_{c_0})$. The main step is then to show that if $\psi$ belongs to the fiber over $(b,w)$ then, for all $f$,  $\psi(f)=\psi(f\circ\Lambda_b\circ S)$. We make use of this equality to derive the vector-valued version of the statement:
	\end{remark}
	
	\begin{lemma} \label{mapping R}
		Let $g_1 \in \mathcal{H}^\infty(B_{c_0})$ such that $g_1(B_{c_0}) \subset \mathbb{D}$. Then for every $h \in \overline{B}_{\mathcal{H}^\infty(B_{c_0},\ell_\infty)}$ we have that $\mathscr{F}(h)$ is Gleason isometric to $\mathscr{F} (g_1,h)$.
	\end{lemma}
	
	\begin{proof}
		Consider the mapping: $R_{g_1}: \mathscr{F}(h) \to \mathscr{F}(g_1,h)$ given by
		\begin{align*}
		R_{g_1}(\Phi)(f)(x) &= \Phi (f\circ \Lambda_{g_1(x)}) (x).
		\end{align*}
		To check that $R_{g_1}$ is a well-defined Gleason isometry take $f \in \mathcal{H}^\infty(B_{c_0})$ and $\Phi \in \mathscr{F}(h)$. For $x \in B_{c_0}$ it is clear that $f\circ  \Lambda_{g_1(x)}$ belongs to $\mathcal{H}^\infty(B_{c_0})$. We need to go a step further and check that $[x \mapsto \Phi (f\circ \Lambda_{g_1(x)}) (x)]$ is a bounded holomorphic mapping in $B_{c_0}$. Since the mapping $[(x,y) \mapsto \Phi (f\circ \Lambda_{g_1(x)})(y)]$ is separately holomorphic for $x$ and $y$ in $B_{c_0}$, by Hartog's Theorem, it is holomorphic as a function of two variables. Thus, it is so when restricted to the diagonal. Additionally, it is clear that this function is bounded by $\|f\|$ and that $R_{g_1}(\Phi)$ is a homomorphism; so  $R_{g_1}(\Phi) \in \mathcal{M}_\infty(B_{c_0},B_{c_0})$.
		Now we want to see that $\xi(R_{g_1}(\Phi)) = (g_1,h)$.  Take $x^* \in \ell_1$ and denote by $S_1$ the left shift operator from $\ell_1$ onto $\ell_1$. Then,
		\begin{align*}
		R_{g_1}(\Phi)(x^*)(x) &= \Phi(x^*\circ \Lambda_{g_1(x)})(x) = \Phi(x^*_1 g_1(x) + S_1x^*)(x) \\
		&= x^*_1g_1(x) + \Phi(S_1 x^*)(x) = x^*_1g_1(x) + h(x)(S_1 x^*)\\
&=x^*(g_1(x),h(x)).
		\end{align*}
		This says that $R_{g_1}$ maps the fiber over $h$ to the fiber over $(g_1,h)$ and thus $R_{g_1}$ is well-defined.	
		Next we  prove that $R_{g_1}$ is onto. For that, let $\Psi \in \mathscr{F}(g_1,h)$ and define the mapping $\Phi \in \mathcal{M}_\infty(B_{c_0},B_{c_0})$ as
		\begin{align*}
			\Phi(f)(x) = \Psi(f\circ S)(x).
		\end{align*}
		Reasoning as before, it is readily seen that $\Phi$ is a homomorphism from $\mathcal{H}^\infty(B_{c_0})$ to $\mathcal{H}^\infty(B_{c_0})$ and $\xi(\Phi) = h$. What is more, for any $x \in B_{c_0}$ and $f \in \mathcal{H}^\infty(B_{c_0})$ we can  apply Remark \ref{TaylorFalco} to $\delta_x \circ \Psi$ obtaining $\delta_x \circ \Psi (f) = \delta_x \circ \Psi (f\circ \Lambda_{g_1(x)}\circ S)$. It follows that
		\begin{align*}
			\delta_x \circ R_{g_1}(\Phi)(f) &= \Phi(f\circ \Lambda_{g_1(x)})(x)\\
			&= \Psi(f\circ \Lambda_{g_1(x)}\circ S)(x)= \delta_x \circ \Psi (f).
		\end{align*}
		So, $R_{g_1}(\Phi)=\Psi$ meaning that $R_{g_1}$ is onto. Finally we derive that $R_{g_1}$ is a Gleason isometry as a consequence of the Gleason isometry commented in Remark \ref{TaylorFalco} for the scalar-valued case and the fact that $S$ and $\Lambda_b$ are linear contractions satisfying $S\circ \Lambda_b=Id$.
	\end{proof}

As in the scalar-valued case, it is important to note that the procedure of the previous lemma can be repeated in several coordinates, not necessarily the first ones. Hence if $h \in \overline{B}_{\mathcal{H}^\infty(B_{c_0},\ell_\infty)}$ and we insert to $(h_n)$ finitely many coordinates $g_1,\dots, g_k$ with $g_i \in \mathcal{H}^\infty(B_{c_0})$ and $g_i(B_{c_0}) \subset \mathbb{D}$ then the fiber over the resulting function is Gleason isometric to the fiber over $h$.

	\section{Gleason Parts} \label{Gleason-parts}

Before get into the subject let us recall some information about Gleason parts of general vector-valued spectra. If $\A$ and $\B$ are uniform algebras with (scalar-valued) spectra $\M(\A)$ and $\M(\B)$ respectively, let us denote by $\M(\A,\B)$ their vector-valued spectrum, that is the set of nonzero algebra homomorphisms from $\A$ into $\B$. For any $\Phi\in\M(\A,\B)$ it is known that its transpose $\Phi^*$ maps $\M(\B)$ into $\M(\A)$.  

Recall that
\[\mathcal {GP}(\Phi) = \{\Psi : \, \|\Phi - \Psi\| < 2\}=\{\Psi : \, \|\Phi^* - \Psi^*\| < 2\}= \{\Psi : \sup_{\varphi\in\M(\B)}\rho(\Phi^*(\varphi), \Psi^*(\varphi))<1\}.
\] An element $\Phi$ in $\M(\mathcal{A},\mathcal{B})$ is then said to be isolated in the Gleason metric if $\mathcal{GP}(\Phi)=\{\Phi\}$.

It is proved in \cite[Th. 6.2]{GalindoGamelinLindstrom} that if $\M(\B)$ is connected and $\Phi^*$ is a non constant  mapping taking strong boundary points of $\M(\B)$ to strong boundary points of $\M(\A)$ then  $\Phi$ forms a singleton \textit{hyperbolic} vicinity in $\M(\A,\B)$. This in turn implies that the Gleason part of $\Phi$ in $\M(\A,\B)$ is a singleton. Since we are focusing solely on the Gleason topology, we can slightly extend the result in the following way. We include the proof for the sake of completeness, even if it is  naturally adapted from the one in \cite[Th. 6.2]{GalindoGamelinLindstrom}.

	\begin{proposition} \label{sbp-to-singleton}
		Let $\A$ and $\B$ be uniform algebras and $\Phi\in\M(\A,\B)$. If $\Phi^*$ maps each strong boundary point of $\M(\B)$ into a singleton Gleason part of $\M(\A)$ then $\Phi$ is isolated in the Gleason metric for $\mathcal{M}(\mathcal{A},\mathcal{B})$.
	\end{proposition}
	\begin{proof}
		Suppose that  $\Psi\in\M(\A,\B)$ is in the same Gleason part as $\Phi$. Then for all $\varphi \in \mathcal{M}(\mathcal{B})$ we have that
	\begin{align*}
		\rho(\varphi \circ\Phi, \varphi \circ\Psi)=\rho(\Phi^*(\varphi), \Psi^*(\varphi)) \leq r < 1.
	\end{align*}
	If $\varphi$ is a strong boundary point our hypothesis tells us that $\Phi^*(\varphi) = \Psi^*(\varphi)$. Thus,
	\begin{align*}
		\varphi (\Phi(f)) = \varphi (\Psi(f)), \ \forall \varphi \textrm{ strong boundary point of } \M(\B),\ \forall f \in \mathcal{A}.
	\end{align*}	
	Since every element of a uniform algebra reaches its norm at a strong boundary point of its spectrum, we derive that $\Phi(f)=\Psi(f)$, for all $f$ and hence $\Phi=\Psi$. 
	\end{proof}

We have begun in \cite{DimantSinger}  studying  relationships about fibers and Gleason parts for the spectrum $\M_\infty(B_X,B_Y)$. Now we want to go deeper for the particular case $\M_\infty(B_{c_0},B_{c_0})$. For the scalar-valued spectrum $\M_\infty(B_{c_0})$ the description of its Gleason parts (and their interaction with fibers) was addressed in \cite{AronDimantLassalleMaestre}. There, the starting point was describing the Gleason parts for the simpler spectrum $\M_u(B_{c_0})$. Here, we follow that path devoting us first to the spectrum $\M_{u,\infty}(B_{c_0},B_{c_0})$.

	\subsection{Gleason parts in $\M_{u,\infty}(B_{c_0},B_{c_0})$.}
	
	The projection $\xi:\M_\infty(B_{c_0},B_{c_0})\to \overline B_{\H^\infty(B_{c_0},\ell_\infty)}$ can be restricted to $\M_{u,\infty}(B_{c_0},B_{c_0})$ maintaining the same image. Also, as finite type polynomials are dense in $\mathcal{A}_u(B_{c_0})$ and the composition homomorphism $C_g$ is defined from $\A_u(B_{c_0})$ to $\H_\infty(B_{c_0})$ for every $g \in \overline{B}_{\mathcal{H}_\infty(B_{c_0},\ell_\infty)}$, an analogous argument to \cite[p. 10]{DiGaMaSe} and \cite[Prop. 2.2]{DimantSinger} yields that $\xi$ is one-to-one. In other words, $\M_{u,\infty}(B_{c_0},B_{c_0})=\{C_g \colon g \in \overline{B}_{\mathcal{H}_\infty(B_{c_0},\ell_\infty)}\}$. That is, each fiber for this spectrum is a singleton. Moreover, for any $x\in B_{c_0}$ and $g \in \overline{B}_{\mathcal{H}_\infty(B_{c_0},\ell_\infty)}$, we know that $\delta_x\circ C_g=\delta_{g(x)}\in \M_u(B_{c_0})$. So, if we denote by $\rho_u$ the pseudo-hyperbolic distance  for the spectrum $\M_u(B_{c_0})$,    we have
	$$	\mathcal{GP}(C_g)  =  \{C_h : \sup_{x \in B_{c_0}} \rho_u(\delta_x \circ C_g, \delta_x \circ C_h) < 1 \} = \{C_h : \sup_{x \in B_{c_0}} \rho_u(\delta_{g(x)},\delta_{h(x)}) < 1 \}
$$
Recall that the equality proved in \cite[Th. 2.4]{AronDimantLassalleMaestre} says (using $\rho_u$ also for the pseudo-hyperbolic distance in the spectrum $\M_u(\D)$)
\begin{equation} \label{rho_u}
   \rho_u(\delta_{g(x)},\delta_{h(x)})= \sup_{n\in \mathbb N} \rho_u(\delta_{g_n(x)},\delta_{h_n(x)}) = \sup_{n\in\mathbb N} \left|\frac{g_n(x)-h_n(x)}{1-\overline{g_n(x)}h_n(x)} \right|, 
\end{equation}
 where the last fraction should be replaced by 0 whenever $g_n(x)=h_n(x)$.
 The properties of the Gleason part $\mathcal{GP}(C_g)$ are then closely related to the properties of the associated  holomorphic function $g$.

 In that regard, the statement of \cite[Prop. 5.1]{DimantSinger} can be easily adapted to our setting so that we obtain, as in the case of $\M_\infty(B_X,B_Y)$ a splitting of the fibers in three sets: interior fibers, middle fibers and edge fibers. Elements from different kind of fibers can not share a  Gleason part. Precisely, the situation here (with singleton fibers) is the following:
 
 \begin{itemize}
     \item Interior fibers are those $C_g$ with $\|g\|<1$. They all share the same Gleason part.
     \item Middle fibers are those $C_g$ with $g(B_{c_0})\subset B_{\ell_\infty}$ and $\|g\|=1$.
     \item Edge fibers are those $C_g$ with $g(B_{c_0})\subset S_{\ell_\infty}$.
 \end{itemize}

Nothing more can be said about \textit{interior fibers} since there is only one Gleason part for all of them. But within \textit{middle fibers} and \textit{edge fibers} we will see that there are a lot of Gleason parts, some of them singleton and some of them containing \textit{balls} of elements. We begin with two known (or easily deduced) examples of singleton Gleason parts.

\begin{example} \label{biholomorphic}
   If $g:B_{c_0}\to B_{c_0}$ is a biholomorphic function then $C_g^*:\M_\infty(B_{c_0})\to \M_u(B_{c_0})$ maps strong boundary points into singleton Gleason parts. Indeed, as it was noted in \cite[p. 15]{DimantSinger}, $C^*_g$ maps strong boundary points to strong boundary points when viewed as a self map of $\M_\infty(B_{c_0})$. Additionally, by \cite[Prop. 3.6]{AronDimantLassalleMaestre}, any strong boundary point in $\M_\infty(B_{c_0})$ is projected over $\mathbb{T}^\infty$. It follows that $C_g^*:\M_\infty(B_{c_0})\to \M_u(B_{c_0})$ maps strong boundary points into singleton Gleason parts.  As a result, any such $C_g$ is a middle fiber isolated in the Gleason metric. Examples of such functions $g$ are for instance $g=Id$ or $g(x)=\left(\frac{a_n-x_n}{1- \overline{a}_nx_n}\right)_n$ with $(a_n)_n\in B_{c_0}$.
\end{example}

\begin{example} \label{edge-constant}
    Let $g:B_{c_0}\to S_{\ell_\infty}$ be a constant function $g(x)=(a_n)_n$ with $|a_n|=1$ for all $n$. Then, it is clear by \eqref{rho_u} that $C_g$ is an edge fiber isolated in the Gleason metric.
\end{example}

In view of the previous examples we wonder if there are other kind of singleton Gleason parts. That is, if there exist isolated edge fibers $C_g$ with $g$ non constant or isolated middle fibers $C_g$ with $g$ non biholomorphic. In order to answer these questions we first present a result that relates singleton Gleason parts of $\M_{u,\infty}(B_{c_0},B_{c_0})$ with singleton Gleason parts of $\M(\A(\D), \H^\infty(B_{c_0}))$ (where $\A(\D)=\A_u(\D)$ is the algebra of holomorphic functions in $\D$ which are continuous in $\overline{\D}$).

	\begin{lemma} \label{coordenadas}
	Let $g\in \overline{B}_{\H^\infty(B_{c_0}, \ell_\infty)}$, $g=(g_n)_n$. Then, $C_g$ is isolated in $\M_{u,\infty}(B_{c_0},B_{c_0})$ if and only if $C_{g_n}$ is isolated in $\M(\A(\D), \H^\infty(B_{c_0}))$ for all $n\in\mathbb N$.
	\end{lemma}

	\begin{proof}
		Let $g \not= h \in \overline{B}_{\mathcal{H}^\infty(B_{c_0},\ell_\infty)}$. If $C_{g_n}$ is isolated for every $n \in \mathbb{N}$, we have that
			\begin{align*}
				\sup_{x \in B_{c_0}} \rho_u(\delta_x \circ C_g, \delta_x \circ C_h) = \sup_{x \in B_{c_0}} \sup_{n \in \mathbb{N}} \rho_u(\delta_x \circ C_{g_n}, \delta_x \circ C_{h_n}) = 1.
			\end{align*}
		Thus, $C_g$ is isolated.
		Conversely, assume that there exists $n_0 \in \mathbb{N}$ such that $C_{g_{n_0}}$ is not isolated in $\mathcal{M}(\mathcal{A}(\D),\mathcal{H}^\infty(B_{c_0}))$. Taking $C_{\widehat{g}_{n_0}} \in \mathcal{GP}(C_{g_{n_0}})$ with $\widehat g_{n_0}\not= g_{n_0}$ and defining $h \in \overline{B}_{\mathcal{H}^\infty(B_{c_0},\ell_\infty)}$ as $h_n = g_n$ for $n \not= n_0$ and $h_{n_0} = \widehat{g}_{n_0}$ gives us an element in $\mathcal{GP}(C_g)$ different than $C_g$.
	\end{proof}

The previous Lemma together with Example \ref{biholomorphic} (which shows that $C_{Id}$ is isolated) implies that the composition operator corresponding to any coordinate function $g_n: B_{c_0} \to \mathbb{D}$, $g_n(x)=x_n$ is isolated in $\mathcal{M}(\mathcal{A}(\D),\mathcal{H}^\infty(B_{c_0}))$. This can also been seen by easily checking that $C_{g_n}$ sends strong boundary points of $\M_\infty(B_{c_0})$ into singleton Gleason parts of $\M(\A(\D))$. Consequently, again by the previous Lemma, any combination of coordinate functions produces an isolated middle fiber:

\begin{example} \label{not-biholomorphic}
    Let $g:B_{c_0}\to B_{\ell_\infty}$ given by $g(x)=(x_1)_n$. Then, $C_g$ is a middle fiber isolated in the Gleason metric, even though $g$ is not biholomorphic. Clearly, any other election of $g$ with $g_n(x)=x_{k_n}$ for all $n$ (and any $k_n$) has an associated composition homomorphism which is isolated in the Gleason metric.
\end{example}

Lemma \ref{coordenadas} is also useful to provide an example of an isolated edge fiber associated to a non constant function. Indeed, it is readily seen, that any M\"obius function of a coordinate $g:B_{c_0}\to \D$, $g(x)=\frac{a-x_n}{1-\overline ax_n}$ (with $|a|<1$) has an associated composition homomorphism which is isolated in $\mathcal{M}(\mathcal{A}(\D),\mathcal{H}^\infty(B_{c_0}))$. Hence, we can combine a sequence of these functions to obtain a singleton edge fiber:

\begin{example} \label{edge-non-constant}
    Let $(a_n)_n$ be an increasing sequence of real numbers with $0<a_n<1$, for all $n$ and $a_n\to 1$. Consider $g:B_{c_0}\to \ell_\infty$ given by $g(x)=\left(\frac{a_n-x_n}{1- a_nx_n}\right)_n$. Then $g(B_{c_0})\subset S_{\ell_\infty}$ and thus $C_g$ is an edge fiber. Moreover, it is clear that $g$ is non constant and that by the previous argument, $C_g$ is isolated in $\M_{u,\infty}(B_{c_0},B_{c_0})$.
\end{example}

Let us now see that the fact that $C_g$ is isolated forces $g$ to be an extreme point of  $\overline{B}_{\mathcal{H}^\infty(B_{c_0},\ell_\infty)}$. This is stated in the following proposition whose proof is modelled after \cite[Ex. 5.5]{DimantSinger} which, in turn, was adapted from \cite[Ex. 2]{MacCluerOhnoZhao}.
	
	\begin{proposition}
		Let $g \in \overline{B}_{\mathcal{H}^\infty(B_{c_0},\ell_\infty)}$ such that $C_g$ is isolated in $\M_{u,\infty}(B_{c_0},B_{c_0})$. Then $g$ is an extreme point of $\overline{B}_{\mathcal{H}^\infty(B_{c_0},\ell_\infty)}$.
	\end{proposition}
		
	\begin{proof}
		Suppose that $g \in \overline{B}_{\mathcal{H}^\infty(B_{c_0},\ell_\infty)}$ is not an extreme point, then there exist $f\not= h \in \overline{B}_{\mathcal{H}^\infty(B_{c_0},\ell_\infty)}$ such that $g = (f+h)/2$. Let $j(x) = g(x) + k(f(x) - h(x))^2$ with $0< k < 1/8$. Recall that $\|j\| = \sup_{n \in \mathbb{N}} \|j_n\|$. Now a quick computation shows that
		\begin{align*}
			\|j_n\| &= \sup_{x \in B_{c_0}} |g_n(x) + k(f_n(x) - h_n(x))^2| = \sup_{x \in B_{c_0}} \left|\frac{f_n(x)+h_n(x)}{2} + k(f_n(x) - h_n(x))^2\right|\\
			&\leq \sup_{x \in B_{c_0}}\left|\frac{f_n(x)+h_n(x)}{2}\right| + k|(f_n(x) - h_n(x))^2|\\ &\leq \sup_{x \in B_{c_0}}\left|\frac{f_n(x)+h_n(x)}{2}\right| + 4k \left(1 - \left|\frac{f_n(x)+h_n(x)}{2}\right|^2\right)\leq 1.
		\end{align*}
		Thus, $j\in \overline{B}_{\mathcal{H}^\infty(B_{c_0},\ell_\infty)}$. 
		To check that $C_j$ is  in the same Gleason part as $C_g$ we compute, for each $n \in \mathbb{N}$,
		\begin{align*}
			\left|\frac{g_n(x) - j_n(x)}{1 -\overline{g_n(x)}j_n(x) }\right| &= \left|\frac{k(f_n(x)-h_n(x))^2}{1 - \frac{\overline{f_n(x)}+ \overline{h_n(x)}}{2}\left(\frac{f_n(x)+h_n(x)}{2} + k(f_n(x) - h_n(x))^2\right)}\right|\\
			&= \frac{k|f_n(x)-h_n(x)|^2}{\left|1- |\frac{f_n(x)+h_n(x)}{2}|^2 - k\left(\frac{\overline{f_n(x)}+ \overline{h_n(x)}}{2}\right)(f_n(x)-h_n(x))^2\right| }\\
			&\leq \frac{k}{1/4 - k} < 1.
		\end{align*}
		So we conclude that $C_j \in \mathcal{GP}(C_g)$ and thus $C_g$ is not isolated.
	\end{proof}

\medskip

\textbf{Open question 1.} Does the reciprocal of the previous proposition hold? That is, for $g \in \overline{B}_{\mathcal{H}^\infty(B_{c_0},\ell_\infty)}$, is it equivalent that $g$ is an extreme point and that $C_g$ is isolated in $\M_{u,\infty}(B_{c_0},B_{c_0})$?

Note that if we change $B_{c_0}$ by $\D$ the answer is yes. Indeed, combining \cite[Cor. 9]{MacCluerOhnoZhao} with  \cite[Th. 4.1]{HozokawaIzuchiZheng} and a classical characterization of extreme points proved in 1957 by Arens, Buck, Carleson, Hoffman and Royden (see \cite[Th. 12]{deleeuw1958extreme} or \cite[Ch. 9]{hoffman2007banach}) it is obtained that $g\in\overline B_{\H^\infty(\D)}$ with $g(\D)\subset \D$ is an extreme point if and only if $C_g$ is isolated in the set of composition operators from $\H^\infty(\D)$ to $\H^\infty(\D)$. This is equivalent to say that any $g\in\overline B_{\H^\infty(\D)}$ is an extreme point if and only if $C_g$ is isolated in  $\M(\A(\D),\H^\infty(\D))$. 

Since it is readily seen that $g=(g_n)_n$ is a extreme point of $\overline{B}_{\mathcal{H}^\infty(B_{c_0},\ell_\infty)}$ if and only if $g_n$ is a extreme point of $\overline{B}_{\mathcal{H}^\infty(B_{c_0})}$ for every $n$, in view of Lemma \ref{coordenadas}  the open question can be reformulated as follows:
is $C_g$ isolated in $\M(\A(\D),\H^\infty(B_{c_0}))$ whenever $g$ is an extreme point of  $\overline{B}_{\mathcal{H}^\infty(B_{c_0})}$?

\medskip

To finish our intended description of Gleason parts for middle/edge fibers, we show examples of Gleason parts containing \textit{balls} of elements.

\begin{example}
    Using that all interior fibers share a Gleason part and that the distance between two homomorphisms is computed by a supremum over the coordinates it is easy to produce examples of \textit{thick} Gleason parts within middle or edge fibers. Indeed, let $g\in S_{\H^\infty(B_{c_0}, \ell_\infty)}$ (i.e.  $C_g$ is a middle or edge fiber) and consider $h\in S_{\H^\infty(B_{c_0}, \ell_\infty)}$, $h=(h_n)_n$ given by $h_{2n}= g_{n}$ and $h_{2n+1}=0$ for $n \in \mathbb{N}$. Now it is clear that $C_h$ is a middle fiber (resp. edge fiber) whether $C_g$ is a middle (resp. edge) fiber. Changing odd coordinates of $h$ by those of any function in $B_{\H^\infty(B_{c_0}, \ell_\infty)}$ we obtain a function whose associated homomorphism is in the same Gleason part as $C_h$. As a result we have that $\mathcal{GP}(C_h)$ contains a Gleason isometric copy  of $\{C_j:\, j\in B_{\H^\infty(B_{c_0}, \ell_\infty)}\}$.

\end{example}

	\subsection{Gleason parts in $\mathcal{M}_\infty(B_{c_0},B_{c_0})$.} 
For this larger spectrum our goal is to obtain some knowledge about how fibers and Gleason parts relate to each other. First of all, note that we  have already produced an interesting outcome about this relationship in Theorem \ref{bola-en-fibra-vectorial}. Indeed, the statement of the theorem, translated to \textit{Gleason parts' language}, is the following:
\begin{center}
    \begin{minipage}{0.9\textwidth}
    For every $g\in\overline B_{\H^\infty(B_{c_0},\ell_\infty)}$, the fiber $\F(g)$ intersects at least $2^c$ different Gleason parts and each of these intersections is \textit{thick} (since it is a Gleason isometric copy of an infinite dimensional ball).\\
\end{minipage}
\end{center}
\medskip
To complete our intended overview we propose the following questions:
\begin{enumerate}
    \item[(Q1)]  Which Gleason parts have elements from different fibers?
    \item[(Q2)]  Which fibers contain singleton Gleason parts?
    \end{enumerate}

Before get into the subject we observe some trivial facts regarding the interaction between Gleason parts for the spectra $\M_{u,\infty}(B_{c_0},B_{c_0})$ and $\mathcal{M}_\infty(B_{c_0},B_{c_0})$.

Whenever $g,h\in \H^\infty(B_{c_0},\ell_\infty)$ map $B_{c_0}$ into $B_{\ell_\infty}$ the composition homomorphisms $C_g$ and $C_h$ can be defined in $\M_{u,\infty}(B_{c_0},B_{c_0})$ as well as in $\mathcal{M}_\infty(B_{c_0},B_{c_0})$. Since the distance between evaluation homomorphisms coincide whether it is computed in $\M_u(B_{c_0})$  or in $\M_\infty(B_{c_0})$ we deduce the same identity in the vector-valued case:
\begin{equation} \label{C_g}
  \sigma(C_g, C_h)=\sigma_u(C_g, C_h)= \sup_{x\in B_{c_0}} \sup_{n\in\mathbb N} \left|\frac{g_n(x)-h_n(x)}{1-\overline{g_n(x)}h_n(x)} \right|.  
\end{equation}

Recall that we split the fibers in three groups: interior, middle and edge fibers. By \cite[Prop. 5.1]{DimantSinger}, all the elements of any Gleason part should belong to the same \textit{kind} of fibers (that is, all interior or all middle or all edge).
For any $g \in \overline B_{\H^\infty(B_{c_0}, \ell_\infty)}$ with $g(B_{c_0})\subset B_{\ell_\infty}$ we have that both the image of the mapping $\Psi_g$ given by Theorem \ref{bola-en-fibra-vectorial} and the corresponding composition homomorphism $C_g$ lie in the fiber $\mathscr{F}(g)$. Additionally, we know that for different values of $\eta \in \beta(\N) \setminus \N$, the images $\Psi_g^\eta(B_{\ell_\infty})$ lie in different Gleason parts. It is also worth remarking that  the images of the mappings $\Psi_g^\eta$ are also disjoint from $\mathcal{GP}(C_g)$. Indeed, this holds by pairing the equality \eqref{deltax} with Remark \ref{gleasondeltaz}.

\medskip
 
Some of the middle and edge fibers are over functions $g$ which satisfy that their associated homomorphisms $C_g$ are isolated in the spectrum $\M_{u,\infty}(B_{c_0},B_{c_0})$. In the previous subsection we have obtained some information about these functions $g$. Let us call them \textit{isolated functions}.

Note that any $\Phi\in \M_\infty(B_{c_0},B_{c_0})$ can be restricted to $\M_{u,\infty}(B_{c_0},B_{c_0})$ obtaining the composition homomorphism $C_{\xi(\Phi)}$. For every $\Phi, \Psi\in \M_\infty(B_{c_0},B_{c_0})$, it is clear  that 
$$
\sigma(\Phi, \Psi)\ge\sigma_u(C_{\xi(\Phi)}, C_{\xi(\Psi)}).
$$
Hence, if $C_{\xi(\Phi)}$ and $C_{\xi(\Psi)}$ do not share a Gleason part for the spectrum $\M_{u,\infty}(B_{c_0},B_{c_0})$ then the same holds for $\Phi$ and $\Psi$ within the spectrum $\mathcal{M}_\infty(B_{c_0},B_{c_0})$. In particular, if  $g\in\overline B_{\H^\infty(B_{c_0},\ell_\infty)}$ is an isolated function then, for every $\Phi\in\F(g)\subset\M_\infty(B_{c_0},B_{c_0})$, the Gleason part of $\Phi$ is contained in $\F(g)$. This is part of the answer of our first guiding question (Q1). We give a complete answer in the following proposition, which is inspired by \cite[Prop. 3.3]{AronDimantLassalleMaestre}.

\begin{proposition} \label{different-fibers}
Let $\Phi\in\F(g)\subset\M_\infty(B_{c_0},B_{c_0})$. The following are equivalent:
\begin{enumerate}
    \item[(i)] $g$ is not an isolated function.
    \item[(ii)] $\mathcal{GP}(\Phi)$ contains elements from different fibers.
\end{enumerate} 
\end{proposition}

\begin{proof}
We just need to prove that $(i)$ implies $(ii)$ since the reciprocal was explained in the paragraph above.
If $g$ is not isolated, by Lemma \ref{coordenadas} there exists  $n$ such that $g_n$ is not isolated. We can assume (without loss of generality) that $g_1$ is not isolated, that is, there exist $h_1\in \overline B_{\H^\infty(B_{c_0})}$ such that $h_1\not = g_1$ and  $C_{g_1}$ and $C_{h_1}$ are in the same Gleason part  in $\M(\A(\D),\H^\infty(B_{c_0})) $. Observe that $g_1(B_{c_0})\subset \D$ (and the same holds for $h_1$) because if this is not the case,  $g_1$ would be a constant function $g_1(x)=\lambda\in\mathbb T$ meaning that $C_{g_1}$ should be isolated.
Denote by $\widehat g=(g_n)_{n\ge 2} \in \overline B_{\H^\infty(B_{c_0},\ell_\infty)}$ and consider, as in Lemma \ref{mapping R}, the mappings 
$$\begin{array}{ccc}
   R_{g_1}: \F(\widehat g)\to \F(g)  & \quad\textrm{and} \quad &  R_{h_1}: \F(\widehat g)\to \F(h_1, \widehat g)\\
    R_{g_1}(\Phi)(f)(x) = \Phi(f\circ \Lambda_{g_1(x)})(x) & \quad\quad \quad & R_{h_1}(\Phi)(f)(x) = \Phi(f\circ \Lambda_{h_1(x)})(x),
\end{array}$$
where $\Lambda_{g_1(x)}(y)=(g_1(x), y)$. As $R_{g_1}$ is surjective there exists $\Psi\in\F(\widehat g)$ such that $R_{g_1}(\Psi)=\Phi$. Let us see that $R_{h_1}(\Psi)$ is in the same Gleason part of $\Phi$ (but in a different fiber):
\begin{align*}
    \|R_{g_1}(\Psi)-R_{h_1}(\Psi)\| &= \sup_{f\in B_{\H^\infty(B_{c_0})}} \|R_{g_1}(\Psi)(f)-R_{h_1}(\Psi)(f)\|\\
    &=\sup_{f\in B_{\H^\infty(B_{c_0})}} \sup_{x\in B_{c_0}} \left| \Psi\left(f\circ\Lambda_{g_1(x)}-f\circ\Lambda_{h_1(x)}\right) (x) \right|\\
    &\le \sup_{f\in B_{\H^\infty(B_{c_0})}} \sup_{x\in B_{c_0}} \left\| f\circ\Lambda_{g_1(x)}-f\circ\Lambda_{h_1(x)} \right\|\\
    &=\sup_{x\in B_{c_0}}\|\delta_{g_1(x)}-\delta_{h_1(x)}\| = \|C_{g_1}-C_{h_1}\|<2.
\end{align*} 
This completes the proof.
\end{proof}

With respect to our second guiding question (Q2), the above proposition tells us that the only places to look for singleton Gleason parts are the fibers over isolated functions $g$. We would like to know whether all those fibers contain singleton Gleason parts. Unfortunately, we do not have a complete answer for this but we can contribute with some steps in that direction.  Let us recall the four examples of isolated functions (within the middle and the edge fibers) presented in the previous subsection.

\begin{itemize}
    \item Any biholomorphic function $g$ is isolated  by Example \ref{biholomorphic}. As it was commented there, $C_g^*$ maps strong boundary points into strong boundary points. Hence, by Proposition \ref{sbp-to-singleton}, the Gleason part of $C_g$ is singleton in $\M_\infty(B_{c_0}, B_{c_0})$.
    \item Example \ref{edge-constant} shows that each constant function $g(x)=\lambda$, with $\lambda\in\mathbb T^\infty$, is isolated. Since this is an \textit{edge} function, the composition homomorphism $C_g$ is not defined in $\M_\infty(B_{c_0}, B_{c_0})$. But we know in this case that there is a singleton Gleason part in the fiber over $g$. Indeed, for the scalar-valued spectrum it is proved in \cite[Prop. 3.6]{AronDimantLassalleMaestre} that in the fiber over any $\lambda\in\mathbb T^\infty$ there is a singleton Gleason part. Note that the scalar-valued spectrum $\M_\infty(B_{c_0})$ is naturally contained in $\M_\infty(B_{c_0}, B_{c_0})$ and each singleton $\varphi\in\M_\infty(B_{c_0})$ is isolated in $\M_\infty(B_{c_0}, B_{c_0})$. This implies that for a constant function $g(x)=\lambda$, with $\lambda\in\mathbb T^\infty$, the fiber $\F(g)\subset \M_\infty(B_{c_0}, B_{c_0})$ contains a singleton Gleason part.
    \item In Example \ref{not-biholomorphic} an isolated \textit{middle} function $g$ not biholomorphic is presented. By Proposition \ref{different-fibers}, the Gleason part of $C_g$ (in $\M_\infty(B_{c_0}, B_{c_0})$) is contained in $\F(g)$ but we do not know if it is singleton.
    \item The isolated function $g$ of Example \ref{edge-non-constant} is a non-constant \textit{edge} function. We could not answer whether it exists a singleton Gleason part within its fiber.
\end{itemize}

In  the scalar-valued case, the isolated elements of $\M_u(B_{c_0})$ are exactly all the evaluations $\delta_\lambda$ with $\lambda\in\mathbb T^\infty$. As we mentioned above, it is proved in \cite[Prop. 3.6]{AronDimantLassalleMaestre} for the spectrum $\M_\infty(B_{c_0})$ that in the fiber over any isolated element there is a singleton Gleason part. For the vector-valued spectrum we could not answer whether an analogous result holds. 
 Thus, we pose another open question.

\medskip

\textbf{Open question 2.} Let $g\in S_{\H^\infty(B_{c_0},\ell_\infty)}$ be an isolated function. Is there a homomorphism $\Phi\in\F(g)\subset \M_\infty(B_{c_0}, B_{c_0})$ such that the Gleason part of $\Phi$ is a singleton?

\medskip

By the previous explanation we can confine this question to isolated non constant edge functions or isolated non biholomorphic middle functions $g$. In the last case it would be interesting to know not only if there is a homomorphism in the fiber over $g$  with singleton Gleason part but also if $C_g$ itself has singleton Gleason part.

To finish, we want to add that even though we could not prove the existence of singleton Gleason parts in the fibers over the functions $g$ from Examples  \ref{not-biholomorphic} and \ref{edge-non-constant} we can construct other isolated functions which are ``middle non biholomorphic'' or ``edge non constant'' whose fibers do contain singleton parts.

\begin{example}
    Let $g_1\in S_{\H^\infty (B_{c_0})}$ be an isolated function for $\M(\A(\D),\H^\infty (B_{c_0}) )$ such that $g_1(B_{c_0})\subset \D$. For instance, $g_1(x)=x_1$ or $g_1(x)=\frac{a-x_1}{1-\overline a x_1}$ with $|a|<1$. Take a biholomorphic mapping $h:B_{c_0}\to B_{c_0}$. Now, from Lemma \ref{mapping R} we know that there is a surjective Gleason isometry $R_{g_1}:\F(h)\to\F(g_1, h) $. Since we comment above that $\mathcal{GP} (C_h)=\{C_h\}$  and we deduce from Lemma \ref{coordenadas} that $(g_1,h)$ is an
 isolated function,  it is clear due to Proposition \ref{different-fibers} that $R_{g_1}(C_h)=C_{(g_1,h)}$ has singleton Gleason part. Note that $(g_1,h)$ is a middle non biholomorphic function.
 
 If, instead, we consider $h\in S_{\H^\infty (B_{c_0}, \ell_\infty)}$ a constant function, $h(x)=\lambda$ with $\lambda\in\mathbb T^\infty$, we have seen above that there is a (scalar-valued) homomorphism $\Phi\in\F(h)$ with singleton Gleason part. Repeating the previous argument it is easily seen that $R_{g_1}(\Phi)\in\F(g_1,h)$ is isolated.  Note that in this case $(g_1,h)$ is a edge non constant isolated function.
\end{example}

\bibliographystyle{abbrv}

\end{document}